\documentclass[12pt,leqno]{amsart}
\usepackage{euscript, amssymb, amsmath, amsthm}
\usepackage{epsfig}
\usepackage{graphicx}
\usepackage{caption}
\usepackage{longtable}
\usepackage{dcolumn}
\usepackage{setspace}
\usepackage[most]{tcolorbox}
\usepackage[colorlinks=true, citecolor=blue]{hyperref}

\numberwithin{equation}{section}

\newtheorem{theo}{Theorem}[section]
\newtheorem{lem}{Lemma}[section]

\newtheorem{Def}[theo]{Definition}
\theoremstyle{remark}
\newtheorem{rem}{Remark}[section]

\newcommand{\ep}{\varepsilon}
\def\R{{\mathbb{R}}}

\def\d{\displaystyle}
\def\e{{\varepsilon}}
\def\p{\partial}

\date{}


\tcbset{
    frame code={}
    center title,
    left=0pt,
    right=0pt,
    top=0pt,
    bottom=0pt,
    colback=gray!10,
    colframe=white,
    width=\dimexpr\textwidth\relax,
    enlarge left by=0mm,
    boxsep=5pt,
    arc=0pt,outer arc=0pt,
    }

\begin{document}

\title[Blow-up of the solution of a damped wave equation - Invariant case]{Blow-up for  wave equation  with the  scale-invariant damping and combined nonlinearities}
\author[M. Hamouda  and M. A. Hamza]{Makram Hamouda$^{1}$ and Mohamed Ali Hamza$^{1}$}
\address{$^{1}$ Basic Sciences Department, Deanship of Preparatory Year and Supporting Studies, P. O. Box 1982, Imam Abdulrahman Bin Faisal University, Dammam, KSA.}

\medskip

\email{mmhamouda@iau.edu.sa (M. Hamouda)} 
\email{mahamza@iau.edu.sa (M.A. Hamza)}

\pagestyle{plain}


\maketitle
\begin{abstract}
In this article, we study the blow-up of the  damped wave equation  in the \textit{scale-invariant case} and in the presence of two nonlinearities. More precisely, we consider the following equation:
\begin{displaymath}
\d u_{tt}-\Delta u+\frac{\mu}{1+t}u_t=|u_t|^p+|u|^q,
\quad \mbox{in}\ \R^N\times[0,\infty),
\end{displaymath}
with small initial data.\\
For $\mu < \frac{N(q-1)}{2}$ and $\mu \in (0,  \mu_*)$, where $\mu_*>0$ is depending on the nonlinearties' powers and the space dimension ($\mu_*$ satisfies $(q-1)\left((N+2\mu_*-1)p-2\right) = 4$), we prove that the wave equation, in this case, behaves like the one without dissipation  ($\mu =0$). Our result completes the previous studies in the case where the dissipation is given by $\frac{\mu}{(1+t)^\beta}u_t; \ \beta >1$ (\cite{LT3}), where, contrary to what we obtain in the present work,  the effect of the damping is not significant in the dynamics. Interestingly, in our case, the influence of the damping term $\frac{\mu}{1+t}u_t$ is important. 
\end{abstract}


\section{Introduction}
\par\quad

We consider the following  family of semilinear damped wave equations
\begin{equation}
\label{G-sys}
\left\{
\begin{array}{l}
\d u_{tt}-\Delta u+\frac{\mu}{(1+t)^\beta}u_t=a|u_t|^p+b|u|^q,
\quad \mbox{in}\ \R^N\times[0,\infty),\\
u(x,0)=\e f(x),\ u_t(x,0)=\e g(x), \quad  x\in\R^N,
\end{array}
\right.
\end{equation}
where $a$ and $b$ are nonnegative constants, $\mu \ge 0$ and $\beta >0$. Moreover, the parameter $\e$ is supposed to be a positive number small enough and $f$ and $g$ are positive functions which are compactly supported on  $B_{\R^N}(0,1)$.

Throughout this article, we suppose that $p, q>1$ and $q \le \frac{2N}{N-2}$ if $N \ge 3$.

It is worth-mentioning that the presence of two nonlinearities in \eqref{G-sys} has an interesting effect on the (global) existence or the nonexistence of the solution of \eqref{G-sys} and its lifespan. Hence, it is natural to study the influence of the nonlinear terms on the behavior of the solution and see whether or not this may produce  a kind of competition between these nonlinearities. \\

It is well-known that in the \textit{scattering} case, $\beta >1$, the solution, $u^L$, of the linear equation corresponding to \eqref{G-sys}, namely
\begin{equation}\label{1.2}
\d u^L_{tt}-\Delta u^L+\frac{\mu}{(1+t)^\beta}u^L_t=0,
\end{equation} 
 behaves like the one of the wave equation without damping ($\mu =0$). In particular, this means  that the damping term does not play any role.  On the other hand, for $\beta <1$, which corresponds to the \textit{effective} case, the solution of the linear equation \eqref{1.2} behaves like the corresponding parabolic equation, namely $\frac{\mu}{(1+t)^\beta}u_t-\Delta u=0$, see e.g. \cite{Wirth1, Wirth2, Wirth3} and the references therein. However, the case $\beta =1$ corresponds to the \textit{scale-invariant} damping. Indeed, the equation \eqref{1.2} is thus invariant under a hyperbolic scaling. The scale-invariant case constitutes, thus, a transition between the parabolic and hyperbolic types. In fact, in this transition, the parameter $\mu$ plays a crucial role in determining the behavior of the solution of \eqref{1.2}, see for example \cite{Wirth1}.

Coming back to \eqref{G-sys} and letting  $\mu = 0$ and 
$(a,b)=(0,1)$, then the equation \eqref{G-sys} reduces to the classical semilinear wave equation which  is somehow related to the Strauss conjecture. This case  gives rise to a critical power, denoted by $q_S$,  which is a solution of the following quadratic equation
\begin{equation}
(N-1)q^2-(N+1)q-2=0,
\end{equation}
and given explicitly by
\begin{equation}
q_S=q_S(N):=\frac{N+1+\sqrt{N^2+10N-7}}{2(N-1)}.
\end{equation}
More precisely,  if $q \le q_S$ then there is no global solution for  \eqref{G-sys} with small initial data, and for $q > q_S$ a global solution exists; see e.g. \cite{John2,Strauss,YZ06,Zhou} among many other references.\\

Now, for the case $\mu = 0$  and $(a,b)=(1,0)$, the Glassey conjecture states that the critical power $p_G$ is given by
\begin{equation}
p_G=p_G(N):=1+\frac{2}{N-1}.
\end{equation}
The above critical value, $p_G$, gives rise to two regions for $p$ ensuring the existence ($p<p_G$) or the nonexistence ($p \ge p_G$) of a global solution; see e.g. \cite{Hidano1,Hidano2,John1,Rammaha,Sideris,Tzvetkov,Zhou1}.\\

The case  $\mu = 0$ and $a, b \neq 0$ (we can assume without loss of generality that $(a,b)=(1,1)$) presents a new phenomenon related to the combined nonlinearities. Indeed, in this case, the powers satisfying   $p \le p_G$ or $q \le q_S$ naturally imply the solution blow-up by a simple adaptation of the proofs in the previous cases $(a,b)=(0,1)$ or $(1,0)$. However, the novelty in the present situation consists in the obtaining of an additional region where the solution blows up. This new region is characterized by the following relationship between $p$ and $q$:
\begin{equation}\label{1.5}
\lambda(p, q, N):=(q-1)\left((N-1)p-2\right) < 4.
\end{equation}
We refer the reader to \cite{Dai,Han-Zhou,Hidano3,Wang} for more details.\\

Now, we focus on the case $\mu > 0$. First, we recall, as mentioned above, the fact that $\beta >1$  in \eqref{G-sys} does not influence the dynamics \cite{LT,LT2,Wakasa}. However, for the scale-invariant case, $\beta =1$, we will see that the situation in the present article is totally different. In fact, for $(a,b)=(0,1)$, it is known in the literature that if the  weak damping coefficient $\mu$ is relatively large, then the equation \eqref{G-sys} (with $(a,b)=(0,1)$) behaves like the corresponding heat equation. Though, if  $\mu$ is small, then the behavior of \eqref{G-sys} is following the one of the corresponding wave equation. More precisely, for $\mu$ small, it was proven, in 
\cite{LTW} and later on  in \cite{Ikeda} with a substantial improvement, that the critical power is moving a bit compared to the case without damping, and hence we have  for 
$$<0<\mu < \frac{N^2+N+2}{N+2} \quad \text{and} \quad 1<q\le q_S(N+\mu),$$
the blow-up of the solution of \eqref{G-sys}.  \\

On the other hand for $\mu > 0$ and $(a,b)=(1,0)$, the authors prove in \cite{LT2} a blow-up result for the solution of \eqref{G-sys} (with $(a,b)=(1,0)$) 
and they give an upper bound of the lifespan. We stress the fact that in this case there is no restriction  for $\mu$  in the the blow-up region for $p$, namely $p \in (1, p_G(N+2\mu))$. 
\\

In this work, we consider the following Cauchy problem for
the scale-invariant wave equation with combined nonlinearities,
\begin{equation}
\label{T-sys}
\left\{
\begin{array}{l}
\d u_{tt}-\Delta u+\frac{\mu}{1+t}u_t=|u_t|^p+|u|^q, 
\quad \mbox{in}\ \R^N\times[0,\infty),\\
u(x,0)=\e f(x),\ u_t(x,0)=\e g(x), \quad  x\in\R^N,
\end{array}
\right.
\end{equation}
where $\mu > 0$, $\ N \ge 1$, $\e>0$ is a sufficiently  small parameter
and  $f,g$ are chosen in the energy space with compact support.

\par
The emphasis in our  work is the study of the Cauchy problem (\ref{T-sys}) for $\mu>0$ and the influence of this parameter on the blow-up result and the lifespan estimate. For the analogous system of (\ref{T-sys}) with  $(\mu/(1+t))u_t$ being replaced by $(\mu/(1+t)^\beta)u_t$ and  $\beta >1$, which corresponds to the scattering case,  Lai and  Takamura proved in \cite{LT3} that, comparing to the wave equation without damping, the scattering damping term has no influence in the dynamics. The situation is totally different in the scale-invariant case ($\beta =1)$ where the effect of the weak damping is significant in the study of global existence or  blow-up of the solution of (\ref{T-sys}). To overcome the difficulty related to this case, we choose in this work to use the technique of multiplier, and, unlike the scattering case, the multiplier here is not bounded as we can see in \eqref{test1} below. Finally, we stress out that the determination of the threshold for $\mu$ and the obtaining  of the analogous of the assumption \eqref{1.5} constitute  the main objectives of this article.

Due to the nature of the problem under consideration here, we notice  some interesting challenges. For example, it is natural to look for the critical value of $\mu$ where this transition holds. Nevertheless,  it is known that, even in the simpler case  $(a,b)=(0,1)$, no critical value is known and it was only  conjectured the existence of such critical value; see e.g. \cite{Ikeda, LTW,Palmieri}. \\

The rest of the article is organized as follows. In Section \ref{sec-main}, after giving a sense to the solution of (\ref{T-sys}) in the energy space, we state the main theorem of our work. Then, we state and prove in Section \ref{aux} some technical lemmas useful  in the proof of the main result which is the subject of Section \ref{proof}.

\section{Main Result}\label{sec-main}
\par
In this section, we will state our main result, but before that, we give the definition of the solution of (\ref{T-sys}) in the corresponding energy space which reads  as follows:
\begin{Def}\label{def1}
 We say that $u$ is a weak  solution of
 (\ref{T-sys}) on $[0,T)$
if
\[
u\in \mathcal{C}([0,T),H^1(\R^N))\cap \mathcal{C}^1([0,T),L^2(\R^N))
\cap \mathcal{C}^1((0,T),L^p(\R^N)),
\]
satisfies, for all $\Phi\in \mathcal{C}_0^{\infty}(\R^N\times[0,T))$ and all $t\in[0,T)$, the following equation:
\begin{equation}
\label{energysol}
\begin{array}{l}
\d\int_{\R^N}u_t(x,t)\Phi(x,t)dx-\int_{\R^N}u_t(x,0)\Phi(x,0)dx \vspace{.2cm}\\
\d -\int_0^t  \int_{\R^N}u_t(x,s)\Phi_t(x,s)dx \,ds+\int_0^t  \int_{\R^N}\nabla u(x,s)\cdot\nabla\Phi(x,s) dx \,ds\vspace{.2cm}\\
\d  +\int_0^t  \int_{\R^N}\frac{\mu}{1+s}u_t(x,s) \Phi_t(x,s)dx \,ds=\int_0^t \int_{\R^N}\left\{|u_t(x,s)|^p+|u(x,s)|^q\right\}\Phi(x,s)dx \,ds.
\end{array}
\end{equation}
\end{Def}
Now, we introduce the following multiplier
\begin{equation}
\label{test1}
\begin{aligned}
m(t):=(1+t)^{\mu}. 
\end{aligned}
\end{equation}
Using the above definition of $m(t)$, we simply observe that
\[
\frac{m'(t)}{m(t)}=\frac{\mu}{1+t}.
\]
Note that the use of  multiplier's technique is  useful for the study
of the nonlinear damped wave equation in our case.

Hence, with the help of  the multiplier $m(t)$, Definition \ref{def1} can be written in  the following equivalent formulation.
\begin{Def}\label{def2}
 We say that $u$ is a weak  solution of
 (\ref{T-sys}) on $[0,T)$
if
\[
u\in \mathcal{C}([0,T),H^1(\R^N))\cap \mathcal{C}^1([0,T),L^2(\R^N))
\cap \mathcal{C}^1((0,T),L^p(\R^N)),
\]
satisfies, for all $\Phi\in \mathcal{C}_0^{\infty}(\R^N\times[0,T))$ and all $t\in[0,T)$, the following equation:
\begin{equation}
\label{energysol}
\begin{array}{l}
m(t)\d\int_{\R^N}u_t(x,t)\Phi(x,t)dx-\int_{\R^N}u_t(x,0)\Phi(x,0)dx \vspace{.2cm}\\
\d -\int_0^t m(s) \int_{\R^N}u_t(x,s)\Phi_t(x,s)dx \,ds+\int_0^t m(s) \int_{\R^N}\nabla u(x,s)\cdot\nabla\Phi(x,s) dx \,ds\vspace{.2cm}\\
\d=\int_0^t m(s) \int_{\R^N}\left\{|u_t(x,s)|^p+|u(x,s)|^q\right\}\Phi(x,s)dx \,ds.
\end{array}
\end{equation}
\end{Def}

The main result of this article is then stated in the following theorem.
\begin{theo}
\label{blowup}
Let $p, q$ and $\mu < \frac{N(q-1)}{2}$ be such that 
\begin{equation}\label{assump}
\lambda(p, q, N+2\mu)<4,
\end{equation}
where the expression of $\lambda$ is given by \eqref{1.5}.\\
Assume that  $f\in H^1(\R^N)$ and $g\in L^2(\R^N)$ are non-negative functions which are compactly supported on  $B_{\R^N}(0,1)$,
and  do not vanish everywhere.
Let $u$ be an energy solution of \eqref{T-sys} on $[0,T_\e)$ such that $\mbox{\rm supp}(u)\ \subset\{(x,t)\in\R^N\times[0,\infty): |x|\le t+1\}$. 
Then, there exists a constant $\e_0=\e_0(f,g,N,p,q,\mu)>0$
such that $T_\e$ verifies
\[
T_\e\leq
 C \,\e^{-\frac{2p(q-1)}{4-\lambda(p, q, N+2\mu)}},
\]
 where $C$ is a positive constant independent of $\e$ and $0<\e\le\e_0$.
\end{theo}

\begin{rem}
Unlike the case with only one nonlinearity ($|u_t(x,s)|^p$ or $|u(x,s)|^q$), one can note, in addition to the two blow-up regions $p \le p_G$ and $q \le q_S$, the obtaining of another blow-up region, characterized by \eqref{1.5}, which is  the result of the interaction of the combined nonlinearities, see \cite{Hidano3}. This observation still holds in our case but with \eqref{1.5} being replaced by \eqref{assump}, otherwise $p_G(N)$ being replaced by $p_G(N+2 \mu)$ and $q_S(N)$  by $q_S(N+\mu)$. 
\end{rem}

\begin{rem}
The assumption \eqref{assump} can be seen as a smallness condition for $\mu$, namely
$\mu \in [0, \mu_*)$ where $\mu= \mu_*$ satisfies  the equality in \eqref{assump}  (otherwise $\mu_*:=\frac{q+1}{p(q-1)} -\frac{N-1}{2}$). 
\end{rem}

\begin{rem}
Note that the results in 
Theorem \ref{blowup} hold true after replacing the linear damping term in \eqref{T-sys}  $\frac{\mu}{1+t} u_t$ by $\mu b(t)u_t$ with $b(t)$ behaving like $(1+t)^{-1}$ as $t$ goes to $\infty$. The proof of this  generalized damping case can be obtained by following the same steps as in the proof of Theorem \ref{blowup} with the necessary modifications.
\end{rem}

\section{Some auxiliary results}\label{aux}
\par

We introduce the following two positive test functions
\begin{equation}
\label{test11}
\psi(x,t):=e^{-t}\phi(x),
\quad
\phi(x):=
\left\{
\begin{array}{ll}
\d\int_{S^{N-1}}e^{x\cdot\omega}d\omega & \mbox{for}\ N\ge2,\vspace{.2cm}\\
e^x+e^{-x} & \mbox{for}\  N=1,
\end{array}
\right.
\end{equation}
which was introduced in Yordanov and Zhang \cite{YZ06}
and admits the following good properties:
\begin{equation*}
\p_t\psi=-\psi,\quad\p_{tt}\psi=\Delta\psi=\psi.
\end{equation*}
Moreover, we have  the following lemma for the function $\psi(x, t)$.
\begin{lem}[\cite{YZ06}]
\label{lem1} Let  $r>1$.
There exists a constant $C=C(N,p,r)>0$ such that
\begin{equation}
\label{psi}
\int_{|x|\leq t+1}\Big(\psi(x,t)\Big)^{r}dx
\leq C(1+t)^{\frac{(2-r)(N-1)}{2r}},
\quad\forall \ t\ge0.
\end{equation}
\end{lem}
\par
As in the non-perturbed case, we define here the functionals that we will use to prove the blow-up criteria later on:
\begin{equation}
\label{F1def}
F_1(t):=\int_{\R^N}u(x, t)\psi(x, t)dx,
\end{equation}
and
\begin{equation}
\label{F2def}
F_2(t):=\int_{\R^N}\p_tu(x, t)\psi(x, t)dx.
\end{equation}
The next two lemmas give the first  lower bounds for $F_1(t)$ and $F_2(t)$, respectively.
\begin{lem}
\label{F1}
Assume that the assumption in Theorem \ref{blowup} holds. Then, we have
\begin{equation}
\label{F1postive}
F_1(t)\ge \frac{\e}{2 m(t)}\int_{\R^N}f(x)\phi(x)dx, 
\quad\text{for all}\ t \in [0,T).
\end{equation}
\end{lem}
\begin{proof} 
Using Definition \ref{def2} and by performing an integration by parts in space in the fourth term in the left-hand side of \eqref{energysol}, we obtain
\begin{equation}\label{eq3}
\begin{array}{l}
\d m(t)\int_{\R^N}u_t(x,t)\Phi(x,t)dx
-\e\int_{\R^N}g(x)\Phi(x,0)dx \vspace{.2cm}\\
\d-\int_0^tm(s)\int_{\R^N}\left\{
u_t(x,s)\Phi_t(x,s)+u(x,s)\Delta\Phi(x,s)\right\}dx \, ds \vspace{.2cm}\\
\d=\int_0^tm(s)\int_{\R^N}\left\{|u_t(x,s)|^p+|u(x,s)|^q\right\}\Phi(x,s)dx \, ds, \quad \forall \ \Phi\in \mathcal{C}_0^{\infty}(\R^N\times[0,T)).
\end{array}
\end{equation}
Now, substituting in \eqref{eq3} $\Phi(x, t)$ by $\psi(x, t)$, we infer that
\begin{equation}
\begin{array}{l}\label{eq4}
\d m(t)\int_{\R^N}u_t(x,t)\psi(x,t)dx
-\e\int_{\R^N}g(x)\psi(x,0)dx \vspace{.2cm}\\
\d+\int_0^tm(s)\int_{\R^N}\left\{
u_t(x,s)\psi(x,s)-u(x,s)\psi(x,s)\right\}dx \, ds \vspace{.2cm}\\
\d=\int_0^tm(s)\int_{\R^N}\left\{|u_t(x,s)|^p+|u(x,s)|^q\right\}\psi(x,s)dx \, ds.
\end{array}
\end{equation}
Using the definition of $F_1$, as in \eqref{F1def}, and the fact that $$\d \int_0^tm(s)F_1'(s) ds=- \int_0^tm'(s)F_1(s) ds+m(t)F_1(t)-F_1(0),$$ the equation  \eqref{eq4} yields
\begin{equation}
\begin{array}{l}\label{eq5}
\d m(t)(F_1'(t)+2F_1(t))
-{\e}C(f,g)  \vspace{.2cm}\\
\d=\int_0^tm'(s)F_1(s) ds+\int_0^tm(s)\int_{\R^N}\left\{|u_t(x,s)|^p+|u(x,s)|^q\right\}\psi(x,s)dx \, ds,
\end{array}
\end{equation}
where 
$$C(f,g)=\int_{\R^N}\left\{f(x)+g(x)\right\}\phi(x)dx.$$
Dividing \eqref{eq5} by $m(t)$ and multiplying the obtained equation by $e^{2t}$, we deduce after integrating  over $[0,t]$ that
\begin{align}\label{F1+}
 e^{2t}F_1(t)
\ge F_1(0)+{\e}C(f,g)\int_0^t \frac{ e^{2s}}{m(s)}ds+  \int_0^t\frac{\mu e^{2s}}{m(s)}\int_{0}^{s} (1+\tau)^{\mu-1}F_1(\tau)d\tau.
\end{align}
Thanks to \eqref{F1+} and the fact that $F_1(0) > 0$, we can easily see that $F_1(t)>0$. Hence, we have
\begin{align}
 F_1(t)
\ge F_1(0)e^{-2t}+{\e}C(f,g)\int_0^t \frac{ e^{2s-2t}}{m(s)}ds.
\end{align}
Remember that $m(t)$, given by \eqref{test1}, is an increasing function (since here $\mu >0$), we get
\begin{align}
 F_1(t)
\ge   F_1(0)e^{-2t}+\frac{{\e}C(f,g)}{2m(t)}(1-e^{-2t})\ge   {{\e}C(f,0)}e^{-2t}+\frac{{\e}C(f,0)}{2m(t)}(1-e^{-2t}).
\end{align}
Finally, using $m(t) \ge 1$, we obtain \eqref{F1postive}. This ends the proof of Lemma \ref{F1}.
\end{proof}

Now we are in a position to prove  the following lemma.
\begin{lem}
\label{F11}
Under the same assumption of Theorem \ref{blowup}, it holds that
\begin{equation}
\label{F2postive}
F_2(t)\ge \frac{\e}{2m(t)}\int_{\R^N}g(x)\phi(x)dx,
\quad\text{for all}\ t \in [0,T).
\end{equation}
\end{lem}
 
\begin{proof}
Let $t \in [0,T)$. Hence, using the definition of $F_1$ and  $F_2$, given respectively by \eqref{F1def} and  \eqref{F2def}, and the fact that
 \begin{equation}\label{def23}\d F_1'(t) +F_1(t)= F_2(t),\end{equation}
 the equation  \eqref{eq5} yields
\begin{equation}
\begin{array}{l}\label{eq5bis}
\d m(t)(F_2(t)+F_1(t))
-{\e}C(f,g)  \vspace{.2cm}\\
\d=\int_0^tm'(s)F_1(s) ds+\int_0^tm(s)\int_{\R^N}\left\{|u_t(x,s)|^p+|u(x,s)|^q\right\}\psi(x,s)dx \, ds.
\end{array}
\end{equation}
Differentiating the  equation \eqref{eq5bis} in time, we obtain 
\begin{align}\label{F1+bis}
\frac{d}{dt} \left\{F_2(t)m(t)\right\}+   m(t)\frac{d}{dt} F_1(t)
= m(t)\int_{\R^N}\left\{|u_t(x,t)|^p+|u(x,t)|^q\right\}\psi(x,t)dx.
\end{align}
Using   \eqref{def23}, the identity \eqref{F1+bis} becomes
\begin{align}\label{F1+bis2}
\frac{d}{dt} \left\{F_2(t)m(t)\right\}+  2 m(t) F_2(t)
&=m(t)  \left\{F_1(t)+ F_2(t)\right\}+\\&+ m(t)\int_{\R^N}\left\{|u_t(x,t)|^p+|u(x,t)|^q\right\}\psi(x,t)dx.\nonumber
\end{align}
Thanks to \eqref{eq5bis} and Lemma \ref{F1},  we can easily see  that $\d m(t)(F_2(t)+F_1(t))
\ge {\e}C(f,g)$. Then, \eqref{F1+bis2} implies that
\begin{align}\label{22mai2}
\frac{d}{dt} \left\{F_2(t)m(t)e^{2t}\right\}\ge  {\e}C(f,g)e^{2t}.
\end{align}
By integrating in time between $0$ and $t$ the inequality \eqref{22mai2}, we
 obtain
\begin{align}\label{22mai3}
F_2(t)m(t)e^{2t}
\ge   F_2(0)+{\e}C(f,g)\int_0^te^{2s} ds\ge   \frac{{\e}C(0,g)}2e^{2t} .
\end{align}
So, by \eqref{22mai3}, we have    \eqref{F2postive}. This concludes the proof of Lemma
\ref{F11}.
\end{proof}

\section{Proof of Theorem \ref{blowup}}\label{proof}
\par\quad
In this section, we will give the proof of the main theorem in this article which states the blow-up result and the  lifespan estimate of the solution of (\ref{T-sys}). For that purpose, we will make use of the lemmas proven in Section \ref{aux}, the multiplier $m(t)$ and a Kato's lemma type. 

Throughout this section, we will denote by $C$  a generic positive constant which may depend on the data ($p,q,\mu,N,f,g$) but not on $\ep$ and of which the value may change from line to line, but, we keep the same notation to make the presentation simpler.
 
First, using the hypotheses in Theorem \ref{blowup}, we recall that $\mbox{\rm supp}(u)\ \subset\{(x,t)\in\R^N\times[0,\infty): |x|\le t+1\}$. \\
Then, we set
\begin{equation}
F(t):=\int_{\R^N}u(x,t)dx.
\end{equation}
Now, by choosing the test function $\phi$ in \eqref{energysol} such that
$\phi\equiv 1$ in $\{(x,s)\in \R^N\times[0,t]:|x|\le s+1\}$\footnote{The choice of a test function $\phi$  which is identically equal to $1$ is possible thanks to the fact that the initial data $f$ and $g$ are supported on $B_{\R^N}(0,1)$.}, we get
\begin{equation}
\label{energysol1}
\begin{array}{l}
\d m(t)\d\int_{\R^N}u_t(x,t)dx-\int_{\R^N}u_t(x,0)dx
=\int_0^t m(s) \int_{\R^N}\left\{|u_t(x,s)|^p+|u(x,s)|^q\right\}dx \,ds.
\end{array}
\end{equation}
Using the definition of $F$, \eqref{energysol1} can be written as
\begin{equation}
\label{F'0ineq}
m(t)F'(t)= F'(0)+ \int_0^t m(s)\int_{\R^N}\left\{|u_t(x,s)|^p+|u(x,s)|^q\right\}dx \,ds.
\end{equation}
Therefore, by dividing \eqref{F'0ineq} by $m(t)$, integrating over $(0,t)$  and using the positivity of $ F(0)$ and $F'(0)$, we infer that
\begin{align}
\label{F'0ineqmmint}
F(t)\geq \int_0^t\frac{1}{m(s)} \int_0^sm(\tau ) \int_{\R^N}\left\{|u_t(x,\tau)|^p+|u(x,\tau)|^q\right\} dx \,d\tau\,ds.
\end{align}
By H\"{o}lder's inequality and the estimates \eqref{psi} and \eqref{F2postive}, we may bound the nonlinear term as follows:
\[
\begin{aligned}
\int_{\R^N}|u_t(x,t)|^pdx&\geq F_2^p(t)\left(\int_{|x|\leq t+1}\Big(\psi(x,t)\Big)^{\frac{p}{p-1}}dx\right)^{-(p-1)}\geq C\e^p(1+t)^{-\mu p -\frac{(N-1)(p-2)}2}.\\
\end{aligned}
\]
Plugging the above inequality  into \eqref{F'0ineqmmint}, we obtain
\begin{align}
F(t)\geq C\e^p\int_0^t(1+s)^{-\mu} \int_0^s (1+\tau)^{-\mu (p-1) -\frac{(N-1)(p-2)}2}d\tau\,ds.
\end{align}
A straightforward computation yields
\begin{equation}
\begin{aligned}\label{F0first}
F(t)
&\geq C\e^p (1+t)^{-\mu p -\frac{p(N-1)}2} t^{N+1}.
\end{aligned}
\end{equation}

On the other hand, we have
\begin{align}
\Big(\int_{\R^N}u(x,s)dx\Big)^q\le \int_{|x|\le t+1}|u(x,s)|^qdx  \Big(\int_{|x|\le t+1}dx\Big)^{q-1}, 
\end{align}
and consequently we deduce that
\begin{align}\label{f0qsup}
F^q(t)\le  \big(t+1 \big)^{N(q-1)}  \int_{|x|\le t+1}|u(x,s)|^q dx.
\end{align}
Now, by differentiating  \eqref{F'0ineq} with respect to time, we obtain 
\begin{equation}
\label{F'0ineq1}
(m(t)F'(t))' \ge   m(t)\int_{\R^N}\left\{|u_t(x,t)|^p+|u(x,t)|^q\right\}dx \ge   m(t)\int_{\R^N}|u(x,t)|^qdx.
\end{equation}
Using \eqref{f0qsup} in \eqref{F'0ineq1}, we infer that
\begin{equation}
\label{F'0ineq2}
\left(m(t)F'(t)\right)' \ge  \frac{F^q(t)}{\big(1+t \big)^{N(q-1)-\mu}}.
\end{equation}
Thanks to \eqref{F'0ineq} we have $m(t)F'_0(t)>0$. Then, multiplying \eqref{F'0ineq2} by $m(t)F'_0(t)$ yields
\begin{equation}
\label{F25nov3}
\left\{\Big(m(t)F'(t)\Big)^2\right\}'
\ge \frac{2\Big(F^{q+1}(t)\Big)'}{(q+1)(1+t)^{N(q-1)-2\mu}}.
\end{equation}
Integrating the above inequality and using $\d\mu < \frac{N(q-1)}{2}$, we have
\begin{equation}
\label{F25nov4}
\Big(m(t)F'(t)\Big)^2
\ge \frac{2F^{q+1}(t)}{(q+1)(1+t)^{N(q-1)-2\mu}} +\left((F'(0))^2-\frac{2F^{q+1}(0)}{(q+1)}\right).
\end{equation}
Observe that the last term in the right-hand side of \eqref{F25nov4} is positive since we consider here small initial data, and more precisely this holds for $\e$ small enough.\\
Hence, \eqref{F25nov4} implies that
\begin{equation}
\label{F25nov6}
\frac{F'(t)}{F^{1+\delta}(t)}
\ge \sqrt{\frac{2}{q+1}} \frac{F^{\frac{q-1}2-\delta}(t)}{(1+t)^{\frac{N(q-1)}2}},
\end{equation}
for $\delta>0$ small enough.\\

Integrating the inequality \eqref{F25nov6} on $[T_0,t]$, for $T_0>1$, and using   \eqref{F0first}, we obtain
\begin{equation}\label{4.14}
\frac1{\delta}\Big (\frac{1}{F^{\delta}(T_0)}-\frac{1}{F^{\delta}(t)}\Big)
\ge \sqrt{\frac{2}{q+1}}  (C\e^p )^{\frac{q-1}2-\delta} \int_{T_0}^t \frac{(1+s)^{(2-\mu p -\frac{(N-1)(p-2)}2)(\frac{q-1}2-\delta)}}{(1+s)^{\frac{N(q-1)}2}}ds.
\end{equation}
Neglecting the second term on the left-hand side in \eqref{4.14} which gives
\begin{equation}\label{eq1/F}
\frac{1}{F^{\delta}(T_0)}
\ge \delta \sqrt{\frac{2}{q+1}}  (C\e^p )^{\frac{q-1}2-\delta} \int_{T_0}^t (1+s)^{-\frac{ \lambda(p, q, N+2\mu)}{4}-\delta\left(2-\mu p -\frac{(N-1)(p-2)}2\right)}ds.  
\end{equation}
Using the hypothesis \eqref{assump}, we have $-\frac{ \lambda(p, q, N+2\mu)}{4}+1>0$. Hence, we can choose $\delta=\delta_0$ small enough such that $\gamma:=-\frac{ \lambda(p, q, N+2\mu)}{4}-\delta_0\left(2-\mu p -\frac{(N-1)(p-2)}2\right)>-1$. Then, the estimate \eqref{eq1/F} yields
\begin{equation}\label{eq1/F-2}
\frac{1}{F^{\delta_0}(T_0)}
\ge C   \e^{\frac{p(q-1)}2-p\delta_0} \, \left((1+t)^{\gamma+1}-(1+T_0)^{\gamma+1}\right).
\end{equation}
Now, using \eqref{F0first} and the fact that $T_0>1$,  we infer that
\begin{equation}\label{eq1/F-3}
\e^{\frac{p(q-1)}2} \, \left((1+t)^{\gamma+1}-(1+T_0)^{\gamma+1}\right) 
\le C (1+T_0)^{-2\delta_0+\mu p \delta_0 +\frac{(N-1)(p-2)\delta_0}2}.
\end{equation}
Consequently, we have 
\begin{eqnarray}\label{eq1/F-4}
\e^{\frac{p(q-1)}2} \, (1+t)^{\gamma+1} 
&\le& C_0 (1+T_0)^{-2\delta_0+\mu p \delta_0 +\frac{(N-1)(p-2)\delta_0}2}\\&&+\e^{\frac{p(q-1)}2} \, (1+T_0)^{\gamma+1},\nonumber
\end{eqnarray}
where $C_0=C_0(p,q,\mu,N,f,g)$.\\
At this level, since $-\frac{ \lambda(p, q, N+2\mu)}{4}+1>0$, then for all  $\e>0$, we choose $T_0>1$  such that
\begin{equation}\label{eq1/F-5}
T_0^{-\frac{ \lambda(p, q, N+2\mu)}{4}+1} = C_0 \e^{-\frac{p(q-1)}2}.
\end{equation}

Hence, using \eqref{eq1/F-5}, we deduce from \eqref{eq1/F-4} that
\begin{equation}\label{eq1/F-6}
t \le 2^{\frac{1}{\gamma+1}}(1+T_0) \le C_1 \e^{-\frac{2p(q-1)}{4-\lambda(p, q, N+2\mu)}},
\end{equation}
where $C_1=C_1(p,q,\mu,N,f,g)$.

This achieves the proof of Theorem \ref{blowup}.\hfill $\Box$

\bibliographystyle{plain}

\begin{thebibliography}{20}



\bibitem{Dai}
{W. Dai, Wei, D. Fang and C. Wang}, {\it 
 Global existence and lifespan for semilinear wave equations with mixed nonlinear terms.} J. Differential Equations, {\bf 267} (2019), no. 5, 3328--3354.

\bibitem{Han-Zhou}
{W. Han and Y. Zhou}, {\it Blow up for some semilinear wave equations in multi-space dimensions.} Comm. Partial Differential Equations, {\bf 39} (2014), no. 4, 65--665. 

\bibitem{Hidano1}
{K. Hidano and K. Tsutaya}, {\it Global existence and asymptotic behavior of solutions for nonlinear wave
equations}, Indiana Univ. Math. J., {\bf 44} (1995), 1273--1305.

\bibitem{Hidano3}
{ K. Hidano, C. Wang and K. Yokoyama}, {\it  Combined effects of two nonlinearities in lifespan of small solutions to semi-linear wave equations.} Math. Ann. {\bf 366} (2016), no. 1-2, 667--694.

\bibitem{Hidano2}
{K. Hidano, C. Wang and K. Yokoyama}, {\it The Glassey conjecture with radially symmetric data}, J. Math.
Pures Appl.,  (9) {\bf 98} (2012),  no. 5, 518--541.


\bibitem{Ikeda}
M. Ikeda and M. Sobajima, 
{\it Life-span of solutions to semilinear wave equation with time-dependent critical damping for specially localized initial data.} 
Math. Ann. {\bf 372} (2018), no. 3-4, 1017--1040.

\bibitem{John1}
{F. John}, {\it Blow-up for quasilinear wave equations in three space dimensions}, Comm. Pure Appl. Math., {\bf 34} (1981), 29--51.


\bibitem{John2}
{F. John}, {\it Fritz Blow-up of solutions of nonlinear wave equations in three space dimensions.} Manuscripta Math. {\bf 28} (1979), no. 1-3, 235--268.

\bibitem{LT}{N.-A. Lai and H. Takamura},
{\it Blow-up for semilinear damped wave equations with subcritical exponent in the scattering case}  Nonlinear Anal. {\bf 168} (2018), 222--237.

\bibitem{LT2}{N.-A. Lai and H. Takamura},
{\it Nonexistence of global solutions of nonlinear wave equations with weak time-dependent damping related to Glassey's conjecture.} 
Differential Integral Equations, {\bf 32} (2019), no. 1-2, 37--48.

\bibitem{LT3}{N.-A. Lai and H. Takamura},
{\it Nonexistence of global solutions of wave equations with weak time-dependent damping and combined nonlinearity.} Nonlinear Anal. Real World Appl. {\bf 45} (2019), 83--96.

\bibitem{LTW}{N.-A. Lai, H. Takamura and K. Wakasa},
{\it Blow-up for semilinear wave equations
with the scale invariant damping and super-Fujita exponent},
J. Differential Equations, {\bf 263(9)} (2017), 5377--5394.


\bibitem{Palmieri}{A. Palmieri},
{\it Alessandro Global existence of solutions for semi-linear wave equation with scale-invariant damping and mass in exponentially weighted spaces.} J. Math. Anal. Appl. {\bf 461} (2018), no. 2, 1215--1240.


\bibitem{Rammaha}{M. A. Rammaha},
{\it Finite-time blow-up for nonlinear wave equations in high dimensions},
Comm. Partial Differential Equations, {\bf 12} (1987), (6), 677--700.




\bibitem{Sideris}
{T. C. Sideris}, {\it Global behavior of solutions to nonlinear wave equations in three space dimensions}, Comm. Partial Differential Equations, {\bf 8} (1983), no. 12, 
1291--1323.

\bibitem{Strauss}
{W. A. Strauss}, {\it  Nonlinear scattering theory at low energy.} J. Functional Analysis, {\bf  41} (1981), no. 1, 110--133.

\bibitem{Tzvetkov}
{N. Tzvetkov}, {\it Existence of global solutions to nonlinear massless Dirac system and wave equation
with small data}, Tsukuba J. Math., {\bf 22} (1998), 193--211.

\bibitem{Wakasa}
{K. Wakasa and B. Yordanov}, {\it  On the nonexistence of global solutions for critical semilinear wave equations with damping in the scattering case. Nonlinear Anal.} {\bf 180} (2019), 67--74.

\bibitem{Wang}{C. Wang and H. Zhou},
{\it   Almost global existence for semilinear wave equations with mixed nonlinearities in four space dimensions.} J. Math. Anal. Appl. {\bf 459} (2018), no. 1, 236--246.
 
\bibitem{Wirth1}{J. Wirth},
{\it Solution representations for a wave equation with weak dissipation},
Math. Methods Appl. Sci., {\bf 27} (2004),  101--124.

\bibitem{Wirth2}{J. Wirth},
{\it Wave equations with time-dependent dissipation. I. Non-effective dissipation},
J. Differential Equations, {\bf 222} (2006), 487--514.

\bibitem{Wirth3}{J. Wirth},
{\it Wave equations with time-dependent dissipation. II. Effective dissipation},
J. Differential Equations, {\bf 232} (2007),  74--103.

\bibitem{YZ06}
{B. Yordanov and Q. S. Zhang}, {\it Finite time blow up for critical wave equations in high dimensions}, J. Funct. Anal., {\bf 231} (2006), 361--374.

\bibitem{Zhou}
{Y. Zhou}, {\it  Blow up of solutions to semilinear wave equations with critical exponent in high dimensions.} Chin. Ann. Math. Ser. B {\bf 28} (2007), no. 2, 205--212.

\bibitem{Zhou1}
{Y. Zhou}, {\it Blow-up of solutions to the Cauchy problem for nonlinear wave equations}, Chin. Ann. Math., {\bf 22B}  (3) (2001), 275--280.




\end{thebibliography}

\end{document}